\documentclass[12pt,reqno]{amsart}
\usepackage[colorlinks=true, pdfstartview=FitV, linkcolor=blue, citecolor=blue, urlcolor=blue]{hyperref}

\usepackage{amscd, amssymb,latexsym,amsmath, amscd, amsmath}
\usepackage[all]{xy}
\usepackage{anysize}
\usepackage[sc]{mathpazo} 
\usepackage{color}
\usepackage{hyperref}
\hypersetup{
    colorlinks=true,
    linkcolor=blue,
    filecolor=magenta,      
    urlcolor=cyan,
    pdftitle={Overleaf Example},
    pdfpagemode=FullScreen,
    }
   
\marginsize{2.5cm}{2.5cm}{2.5cm}{2.5cm}
\newtheorem{theorem}{Theorem}[section] %(If you want theorem numbered
\newtheorem{lemma}[theorem]{Lemma}%               with section number.  Same
\theoremstyle{definition}

\theoremstyle{remark}
\newtheorem{remark}{Remark}

\newcommand{\mbb}{\mathbb}
\newcommand{\de}{\delta}
\newcommand{\ov}{\overline}

\newcommand{\pa}{\partial}
\newcommand{\mf}{\mathbb}

\newcommand{\Om}{\Omega}
\newcommand{\al}{\alpha}
\newcommand{\be}{\beta}
\newcommand{\ga}{\gamma}
\newcommand{\z}{\zeta}
\newcommand{\la}{\lambda}

\newcommand{\Ga}{\Gamma}
\newcommand{\diag}{\operatorname{diag}}

\newcommand{\Ric}{\operatorname{Ric}}
\renewcommand{\Re}{\operatorname{Re}}
\renewcommand{\Im}{\operatorname{Im}}

\begin{document}

\title[ Fefferman--Szeg\"o metric]{Boundary behaviour of the Fefferman--Szeg\"o metric in strictly pseudoconvex domains}
\author[Anjali Bhatnagar]{Anjali Bhatnagar}
\address{
Department of Mathematics, Indian Institute for Science Education and Research, Dr Homi Bhabha Road, Pashan, Pune 411 008, India. }
\email{anjali.bhatnagar@students.iiserpune.ac.in}

\begin{abstract}
We study the boundary behaviour of the Fefferman--Szeg\"o metric and several associated invariants in a $C^\infty$-smoothly bounded strictly pseudoconvex domain.
\end{abstract}

\date{}

\maketitle

%------------------------------------------------------

\section{Introduction}
The Bergman and Szeg\"o kernels are fundamental reproducing kernels in complex analysis. They have been extensively explored in both one and higher dimensions, leading to several applications across various areas of mathematics. The Bergman metric, induced by the Bergman kernel, is biholomorphically invariant, whereas the Szeg\"o metric, defined analogously using the Szeg\"o kernel, lacks this invariance. This discrepancy arises because the Euclidean surface area measure does not generally transform well under biholomorphic mappings, except in one dimension. For many years, it has been commonly believed that there is no analogous theory for the Szeg\"o kernel in higher dimensions.

Fefferman \cite{fc79} introduced the new surface area measure on $C^\infty$-smoothly bounded strictly pseudoconvex domains, now known as the Fefferman surface area measure. Using this, Barrett-Lee \cite{bl14} defined an invariant version of the Szeg\"o metric---called the Fefferman--Szeg\"o metric---and investigated its relationship with classical metrics such as the Bergman and Carath\'eodory metrics. Krantz studied the representative coordinates associated with the Fefferman--Szeg\"o metric, its analytic continuation, and completeness in \cite{k19}. It has been further investigated in \cite{k21}.

In one dimension, the Fefferman–Szeg\"o metric is studied in \cite{bb24} by examining its intrinsic properties, such as geodesics, curvature, and $L^2$-cohomology. A formula for the Fefferman--Szeg\"o metric on annulus is derived in terms of the Weierstrass $\wp$-function. Additionally, the optimality of the universal upper and lower bounds for its Gaussian curvature is investigated. It is also shown that there exist $C^\infty$-smoothly bounded planar domains where the Gaussian curvature of the Fefferman–Szeg\"o metric attains both positive and negative values. Moreover, results from \cite{bb24, z10} indicate that there are domains in which the curvatures of the Bergman and Fefferman–Szeg\"o metrics have opposite signs.

Recently, the existence of closed geodesics and geodesic spirals for the Szeg\"o metric has been established in a 
$C^\infty$-smoothly bounded strictly pseudoconvex domain 
$\Om\subset\mathbb C^n$, for $n\geq 2$, extending a result from \cite{bb24}---which was explored by Herbort for the Bergman metric in \cite{hg83}. Similar techniques can be employed to demonstrate the existence of closed geodesics and geodesic spirals for the Fefferman–Szeg\"o metric as well.

 The purpose of the article is to investigate the boundary behaviour of the invariants associated with the Fefferman--Szeg\"o metric in $C^\infty$-smoothly bounded strictly pseudoconvex domains in the paradigm of scaling, without relying on the localization principle. The boundary behaviour of the Bergman metric (cf. \cite{c02, df18, kgy96, BV-ns, j23, j24}, etc) is well-known on some pseudoconvex domains, and the results in this article demonstrate the similarities between the Bergman and Fefferman–Szeg\"o metrics in $C^\infty$-smoothly bounded strictly pseudoconvex domains. To set the stage, let $\Om=\{z\in\mf{C}^n: r(z)<0\}\subset\mf{C}^n$ be a $C^{\infty}$-smoothly bounded strictly pseudoconvex domain with a $C^{\infty}$-smooth strictly plurisubharmonic defining function $r$ of $\Om$. Let $\sigma_{F}$ be the Fefferman measure defined as follows 
\begin{equation}\label{se}
    d\sigma_{F}\wedge d r=c_{n}\left(-\det\begin{pmatrix}
    0 & r_{\overline{j}}\\
    r_{i} & r_{i\overline{j}}
\end{pmatrix}_{1\leq i, j\leq n}\right)^{\frac{1}{n+1}}dV
\end{equation}
or equivalently,
\begin{equation}
    d\sigma_{F}=c_{n}\left(-\det\begin{pmatrix}
    0 & r_{\overline{j}}\\
    r_{i} & r_{i\overline{j}}
\end{pmatrix}_{1\leq i, j\leq n}\right)^{\frac{1}{n+1}}\frac{d\sigma_{E}}{||d\sigma_{E}||},
\end{equation}
where $\sigma_{E}$ is the Euclidean surface area measure and $r_{i}=\frac{\partial r}{\partial z_{i}}, r_{\overline{j}}=\frac{\partial r}{\partial\overline{z}_{j}}, r_{i\overline{j}}=\frac{\partial^{2}r}{\partial z_{i}\partial\overline{z}_{j}}$. 

  It can be seen that $d\sigma_{F}$ is independent of the choice of defining function $r$ by letting $\Tilde{r}=h r$, where $h$ is a $C^{\infty}$-smooth positive function, and computing $d\sigma_{F}$ with $d\Tilde{r}$. The dimensional constant $c_{n}$ initially undefined in \cite{fc79} has been assigned various values for convenience in different contexts. For instance, in \cite{b06}, it is set as $22n/(n+1)$, while in \cite{hk06}, it is equal to $1$. 
   
   Let $L^{2}(\pa \Om)$ be the space of square-integrable measurable functions with respect to $\sigma_{F}$. The \textit{Hardy space} $ H^{2}(\pa\Om)$ is defined as the closure in $L^{2}(\pa \Om)$ of the space of boundary values of holomorphic functions on $\Om$, which are continuous on $\ov \Om$. Then, there exists a Szeg\"o kernel $S_{\Om}(z, w)$ which is uniquely determined  by the following: for each $z\in \Om, S_{\Om}(\cdot, z)\in H^{2}(\pa\Om)$; for all $z, w\in \Om, S_{\Om}(z, w)=\ov{S_{\Om}(w, z)}$; and for each $h\in H^{2}(\pa\Om),$
\[h(z)=\int_{\pa\Om}h(w)S_{\Om}(z, w)d\sigma_{F}\quad\text{ for all }z\in\Om.\]
It can be seen that $S_{\Om}(z, w)$ can be expressed in terms of any complete orthonormal basis $\{\phi_{j}\}_{k\in\mf{N}}$ of $H^{2}(\pa \Om)$ as follows 
\[S_{\Om}(z,w)=\sum_{j=1}\phi_{j}(z)\ov{\phi_{j}(w)},\]
where the series converges uniformly on compact subsets of $\Om\times \Om$. Moreover, if $F: \Omega_1\to \Omega_2$ is a biholomorphism between two $C^{\infty}$-smoothly bounded strictly pseudoconvex domains in $\mf{C}^n$ such that a well-defined holomorphic branch of $(\det J_{\mathbb{C}}F(z))^{\frac{n}{n+1}}$ on $\Omega_{1}$ exists, then the Szeg\"o kernels satisfies 
\begin{equation}
    S_{\Om_1}(z, w)=S_{\Om_2}(F(z), F(w))(\det J_{\mf{C}}(F(z)))^{\frac{n}{n+1}}(\ov{\det J_{\mf{C}}(F(w))})^{\frac{n}{n+1}}.
\end{equation}

Let $S_{\Omega}(z)=S_{\Omega}(z, z)$ denote the diagonal values of the Szeg\"o kernel. It is known that $\log S_\Om(z)$ is a $C^\infty$-smooth strictly plurisubharmonic function on $\Om$, and therefore induces the K\"ahler metric called the Fefferman--Szeg\"o metric defined as
 \begin{equation}\label{fsm}
      d s_{\Om}^{2}=\sum_{\al,\be=1}^{n}g_{\al\overline{\be}}(z)dz_{\al}d\ov{z}_{\be},
   \end{equation}
   where \[g_{\al\overline{\be}}(z)=\dfrac{\partial^{2}\log S_{\Om}(z)}{\partial z_{\al}\partial\overline{z}_{\be}}.\] 
 Furthermore, if $F:\Om_1\to\Om_2$ is a biholomorphism such that a well-defined holomorphic branch of $(\det J_{\mathbb{C}}F(z))^{\frac{n}{n+1}}$ on $\Omega_{1}$ exists, then
\begin{equation}\label{tr-G}
G_{\Omega_1}(z)=J_\mathbb C F(z)^t\, G_{\Omega_2}\big(F(z)\big) \overline {J_\mathbb C F(z)},
\end{equation}
 where $G_{\Om}(z)=\Big[g_{\al\overline{\be}}(z)\Big]_{\al,\be=1}^{n}$, thus $ds^2_\Om$ is an invariant metric. Let $\tau_{\Om}(z, X)$ be the length of
a vector $X\in\mf{C}^n$ at $z\in\Om$ in $d s_{\Om}^{2}$. The Riemannian volume
element of $d s_{\Om}^{2}$ is denoted by
\[ g_{\Om}(z)=\det G_{\Om}(z),\]
and it induces a biholomorphic invariant as follows
\[\be_{\Om}(z)=\frac{g_{\Om}(z)}{S_{\Om}(z)^{\frac{n+1}{n}}}.\]  
The holomorphic sectional curvature of $ds_{\Om}^{2}$ is given by 
\[R_{\Om}(z, X)=\frac{\sum_{\alpha,\be,\ga,\de}R_{\ov{\al}\be\ga\ov{\de}}(z)\ov{X}^{\al}X^{\be}X^{\ga}\ov{X}^{\de}}{\left(\sum_{\al,\be}g_{\al\ov\be}(z)X^{\al}\ov{X}^{\be}\right)^2},\]
where \[R_{\ov{\al}\be\ga\ov{\de}}(z)=-\frac{\pa^2g_{\be\ov{\al}}}{\pa z_{\ga}\pa\ov{z}_{\de}}+\sum_{\mu,\nu}g^{\nu\ov{\mu}}\frac{\partial g_{\be\ov{\mu}}}{\pa z_{\ga}}\frac{\partial g_{\nu\ov{\al}}}{\pa \ov{z}_{\de}}.\]
Here, $g^{\nu\ov{\mu}}(z)$ being the $(\nu,\mu)$-th entry of the inverse of the matrix $G_{\Om}(z)$. For all $z\in\Om,~X\in \mf{C}^n$, it follows from arguments similar to those used for the Bergman metric that
\[R_{\Om}(z, X)<2.\]
 Bhatnagar-Borah \cite{bb24} proved that $2$ is the sharp, universal upper bound for the Fefferman–Szeg\"o metric, while no universal lower bound exists.

The Ricci curvature of $ds_{\Om}^{2}$ is given by 
\[\text{Ric}_{\Om}(z, X)=\frac{\sum_{\al,\be}\text{Ric}_{\al\ov{\be}}^{\Om}(z)X^{\al}\ov{X}^{\be}}{\sum_{\al,\be}g_{\al\ov{\be}}(z)X^{\al}\ov{X}^{\be}},\]
where \[\text{Ric}_{\al\ov{\be}}^{\Om}(z)=-\frac{\pa^2}{\pa z_{\al}\pa\ov{z}_{\be}}\log g_{\Om}(z).\]
To state the main results of the article, we recall the following. let $\de_\Om(z)$ represent the Euclidean distance from a point 
$z\in\Om$ to $\pa\Om$. For points $z$ sufficiently close to $\pa\Om$, let $\pi(z)\in\pa\Om$ be the closest point to $z$ such that $\de_\Om(z)=\vert z-\pi(z)\vert$. Additionally, for a tangent vector $X\in\mf C^n$ at $z$, we can express $X$ as $X=X_H(z)+X_N(z)$, where $X_H(z)$ and $X_N(z)$ are the components of $X$ along the tangential and normal directions at $\pi(z)$, respectively. Finally, let $\mathcal{L}_{\pa\Om}$ denote the Levi form 
 with respect to some defining function for $\Om$.

\begin{theorem}\label{bdy-n}
    Let $\Omega \subset \mathbb{C}^n$ be a $C^\infty$-smoothly bounded strictly pseudoconvex domain, and let $p^0 \in \partial \Omega$. Then as $z\to p^0$, we have
\begin{itemize}
\item[(a)] $\delta_\Omega(z)^{n+1} g_{\Omega}(z) \to \dfrac{n^n}{2^{n+1}}$,
\item[(a)] $\delta_{\Omega}(z)\, ds_{\Omega}(z,X_N(z))\to \dfrac{\sqrt{n}}{2}\,\vert X_N(p^0)\vert$,
\item[(c)]$\sqrt{\delta_{\Omega}(z)}\, ds_{\Omega}\big(z,X_{H}(z)\big)\to \sqrt{\dfrac{n}{2}\mathcal{L}_{\partial \Omega}\big(p,X_H(p^0)\big)}$,
\item[(d)] $\beta_{\Omega}(z) \to \left(\dfrac{c_n}{(n-1)!}\right)^{\frac{n+1}{n}}n^{n} \pi^{n+1},$
\item[(e)] $R_{\Omega}(z,X)\to -\dfrac{2}{n},\quad X\in\mathbb{C}^n\setminus\{0\}, $
\item[(f)] $\Ric_{\Omega}(z,X)\to -\dfrac{1}{n},\quad X\in\mathbb{C}^n\setminus\{0\}.$ 
\end{itemize}
 
\end{theorem}
Parts (b) and (c) of Theorem~\ref{bdy-n} can be interpreted as counterparts to Graham's result \cite{Gr} for the Kobayashi and Carath\' {e}odory metrics. Furthermore, using \cite[Theorem~1.17]{gk89} and Theorem~\ref{bdy-n}~(e) implies that every isometry between $C^\infty$-smoothly bounded strictly pseudoconvex domains equipped with the Fefferman--Szeg\"o metric is either holomorphic or conjugate holomorphic.
\begin{theorem}\label{loc}
Let $\Om^0\subset\Omega \subset \mathbb{C}^n$ be a $C^\infty$-smoothly bounded strictly pseudoconvex domains, which share an open neighbourhood $\Ga\subset\pa\Om$ near a boundary point $p^0\in \partial \Omega$. Then, for a sufficiently small neighbourhood $U$ of $p^0$, we have
\begin{itemize}
    \item [(a)] $\lim_{z \to p^0}\dfrac{g_{\Omega^0}(z)}{g_{\Omega}(z)}=1$,
    \item[(b)] $\lim_{z \to p^0} \dfrac{\beta_{\Omega^0}(z)}{\beta_{\Omega}(z)}=1$,
    \item[(c)] $ \lim_{z \to p^0}\dfrac{ds_{\Omega^0}(z,X)}{ds_{\Omega}(z,X)}=1$,
    \item[(d)] $\lim_{z \to p^0} \dfrac{2-R_{\Omega^0}(z,X)}{2-R_{\Omega}(z,X)}=1$,
    \item[(e)] $\lim_{z \to p^0} \dfrac{n+1-\Ric_{\Omega^0}(z,X)}{n+1-\Ric_{\Omega}(z,X)}=1$,
\end{itemize}
uniformly on $\{ \Vert X \Vert = 1\}$.
\end{theorem}
\textbf{Acknowledgements.} The author expresses gratitude to D. Borah for suggesting the project.
\section{Stablilty of Szeg\"o kernel under Pinchuk scaling method}
In this section, we prove the Ramadanov-type theorems for the Szeg\"o kernel. We begin by recalling the Pinchuk scaling method from \cite{pin80}. Let $p^0=0\in \partial \Omega$, and let $\Om^0\subset\Om$ be $C^\infty$-smoothly bounded strictly pseudoconvex domains that share an open piece $\Ga\subset\pa\Om$ near $0$. There exists a sufficiently small neighbourhood $U$ of $0$, such that $U\cap\Om^0=U\cap\Om$, along with a $C^\infty$-smooth local defining function $r: U\to\mathbb R$ for $\Omega$. The gradient of $r$ satisfies $\nabla r = 2 \nabla_{\overline{z}} r$ where \( \nabla_z r = \left(\partial r/\partial z_1, \ldots,\partial r/\partial z_n\right) \) and  \( \nabla_{\overline{z}} r = \overline{\nabla_z r} \). We will write $z\in\mf C^n$ as $z=('z,z_n)\in\mf C^{n-1}\times \mf C$. Let
\begin{equation}\label{normal-Re-z_n}
\nabla_{\overline z} r(0)=({'}0,1) \text{ and } \frac{\partial r}{\partial z_n}(z)\neq 0 \text{ for all } z\in U.
\end{equation}

\noindent By the strict pseudoconvexity of $\partial\Omega$, local holomorphic coordinates $z_1,\ldots, z_n$ can be chosen in a sufficiently small neighborhood $U$ of $0$ such that
\begin{equation}\label{initial-df}
r(z)=2\Re z_n+|'z|^2+o\big(\Im z_n,|'z|^2\big), \quad z \in U,
\end{equation}
and \begin{align}\label{subset1}
U \cap \Omega\subset D=\big\{z \in \mathbb C^n: 2\Re z_n+c|'z|^2<0\big\},
\end{align}
where $0<c<1$ is constant. Let $ \z^j $ be a sequence of points in $\Om $ converging to $ 0 \in \partial \Om $. Without loss of generality, for each $j$, there exists a unique $p^j\in U\cap\partial \Omega$ such that $|\zeta^j-p^j|= \delta_{\Omega}(\zeta^j)=\delta_{\Omega^0}(\zeta^j)=\delta_j$. Thus, as $j\to \infty$, $p^j\to 0$ and $\delta_j\to 0$. By 
\cite{pin80}, there exists a sequence $ \{ \psi_{p^j}\} $ of automorphisms of $ \mathbb{C}^n $, such that $q^j =\psi_{p^j} ( \z^j)= ('0, t|\nabla r_z(p^j)|^2)$ and $\psi_{p^j}(p^j)=0,~\psi_{p^0}=i_n$---the identity map of $\mathbb C^n$, along with the domains $ \psi_{p^j}(U\cap\Om) =\Om_j'$ near $0$ are given by
\begin{equation}\label{rj}
    \{ z =('z, z_n) \in \mathbb{C}^n: r_j(z)= 2 \Re \left(z_n + G_j(z) \right) + L_j(z) + o ( |z|^2 ) < 0 \},
\end{equation}
where $ G_j(z) = \sum_{\mu, \nu=1}^n a_{\mu \nu} (p^j) z^{\mu} z^{\nu} $, $ L_j(z) = \sum_{\mu, \nu=1}^n a_{\mu \ov \nu} (p^j) z^{\mu} \overline{z}^{\nu} $ which satisfies
$ G_j('z, 0) \equiv 0 $ and $ L_j('z, 0) \equiv |'z|^2 $. 

Moreover, since $\partial \Omega$ is strictly pseudoconvex, shrinking $U$ if necessary and adjust $c_0$ in \eqref{subset1} to ensure that
\begin{align}\label{subset2}
\psi_{p^j}(U\cap \Omega) \subset D,
\end{align}
 for all $j$ large.

We briefly recall the construction of $\psi_{p^j}$ as its definition and derivative plays an important role in studying the boundary behaviour of invariant objects. For each $p^j\in U\cap\pa\Om$, $\psi_{p^j}=\phi^{p^j}_3\circ\phi^{p^j}_2\circ\phi^{p^j}_1$ where each $\phi^{p^j}_i$ is an automorphism of $\mf C^n$.

 The map $w=\phi_{1}^{p^j}(z)$ is an affine transformation defined as 
 \begin{equation}\label{phi_1-defn}
     \phi_{1}^{p^j}(z)=P_{p^j}(z-p^j),
 \end{equation}
 where  $P_{p^j}:\mathbb C^n\to \mathbb C^n$ is the linear map whose matrix is
\begin{align}\label{P-matrix}
\mathsf{P}_{p^j}=\begin{pmatrix}
\frac{\partial r(p^j)}{\partial \overline z_n} & 0 & \cdots & 0 & -\frac{\partial r(p^j)}{\partial \overline z_1}\\
0 & \frac{\partial r(p^j)}{\partial \overline z_n} & \cdots & 0 & -\frac{\partial r(p^j)}{\partial \overline z_2}\\
\vdots & \vdots & \cdots &\vdots & \vdots\\
0 & 0 & \cdots & \frac{\partial r(p^j)}{\partial \overline z_n} & -\frac{\partial r(p^j)}{\partial \overline z_{n-1}}\\
\frac{\partial r(p^j)}{\partial z_1} & \frac{\partial r(p^j)}{\partial z_2} & \cdots & \frac{\partial r(p^j)}{\partial z_{n-1}} & \frac{\partial r(p^j)}{\partial z_n}
\end{pmatrix}.
\end{align}
 By (\ref{normal-Re-z_n}), $P_{p^j}$ is nonsingular, and
\begin{equation}\label{phi1-n_z}
\Phi_{1}^{p} \big(p+t\nabla_{\overline z}r(p)\big) =t\mathsf{P}  _p\nabla_{\overline z}r(p)=\big({'0}, t\vert \nabla_{\overline z}r(p)\vert^2\big).
\end{equation}
Here and in the sequel, the matrix representation of a linear map \( L: \mathbb{C}^n \to \mathbb{C}^n \) is denoted by \( \mathsf{L} \). Observe that as $p^j\to p^0=0,\;\phi_{1}^{p^j}(z)$ converges to $\phi_{1}^{0}(z)$ uniformly on compact subsets of $\mathbb C^n$. Since $(\phi_{1}^{p^j})'(z)= \mathsf P_{p^j}$, $(\phi_{1}^{p^j})'(z)\to(\phi_{1}^{0})'(z)$ in the operator norm, and uniformly in $z\in\mathbb{C}^n$. By relabelling the new coordinates $w$ by $z$, the local defining function $r\circ(\phi_1^{p^j})^{-1}$ of $\Om_j^1=\phi^{p^j}_1(\Om)$ near $0$ is given by 
\begin{equation}\label{r-1}
r_{1}^j(z)=2\text{Re}\left(z_n+\sum_{\mu,\nu=1}^{n} a^1_{\mu \nu}(p^j) z_{\mu} z_{\nu}\right)+L_{j}^1(z)+o(\vert z \vert^2),
\end{equation}
where
\begin{equation}\label{a1-b1}
\begin{aligned}
G_j^1(z) &=\sum_{\mu,\nu=1}^{n} a^1_{\mu \nu}(p^j) z_{\mu} z_{\nu}, \quad \begin{pmatrix}a^1_{\mu \nu}(p^j)\end{pmatrix}=\frac{1}{2}\Big(\mathsf{P}_{p^j}^{-1}\Big)^t \begin{pmatrix} \frac{\partial^2 r(p^j)}{\partial z_{\mu} \partial z_{\nu}} \end{pmatrix}\mathsf{P}_{p^j}^{-1},\\
L_p^{1}(z) & =\sum_{\mu,\nu=1}^n b^1_{\mu \overline \nu}(p^j) z_{\mu} \overline z_{\nu}, \quad 
\begin{pmatrix}b^1_{\mu\overline \nu}(p^j)\end{pmatrix}=\Big(\mathsf{P}_{p^j}^{-1}\Big)^{*}\begin{pmatrix} \frac{\partial^2 r(p^j)}{\partial z_{\mu}\partial \overline z_{\nu}} \end{pmatrix}\mathsf{P}_{p^j}^{-1}.
\end{aligned}
\end{equation}

 The map $w=\phi_{2}^{p^j}(z)$ is a polynomial automorphism, defined as 
\begin{equation}\label{phi_2-defn}
w=\left('z, z_n+\sum\limits_{\mu,\nu=1}^{n-1}a_{\mu\nu}^1(p^j)z_{\mu}z_{\nu}\right).
\end{equation}
It fixes points on the $\Re z_n$-axis. From (\ref{a1-b1}), as \( p^j \to 0 \), \( \phi_{2}^{p^j}(z) \to \phi_{2}^{0}(z) \) uniformly on compact subsets of \( \mathbb{C}^n \). Moreover,
\begin{equation}\label{der-phi2}
(\phi_2^{p^j})'(z)= \begin{pmatrix}\mathsf{I}_{n-1} & 0\\
\begin{pmatrix}\displaystyle \sum_{\mu=1}^{n-1} a^1_{\mu \ga}(p^j) z_{\mu} + \displaystyle \sum_{\nu=1}^{n-1} a^1_{\ga\nu}(p^j) z_{\nu}\end{pmatrix}_{\ga=1, \ldots, n-1}& 1 \end{pmatrix}.
\end{equation}
Thus, as \( j \to \infty \),~\( (\phi_{2}^{p^j})'(z) \to (\phi_{2}^{0})'(z) \) in operator norm as well as uniformly on compact subsets of \( \mathbb{C}^n \). Also, note that
\begin{equation}\label{der-phi2-zetaj}
    \big(\phi_2^{p^j}\big)'\big(\phi_1^{p^j}(\z^j)\big)=\mathsf I_n.
\end{equation}

Relabeling the coordinates $w$ as $z$, the local defining function $ r \circ (\phi_{1}^{p^j})^{-1} \circ (\phi_{2}^{p^j})^{-1} $ of the domain $\Omega_{j}^2= \phi_{2}^{p^j}\circ \phi_{1}^{p^j}(\Omega)$ near $0$ takes the form \eqref{r-1} with $
a_{\mu\nu}^1(p^j)=0$ for $1\leq \mu, \nu \leq n-1$.

The map $\phi^{p^j}_3$ is selected so that the Hermitian form $L_j^1(z)$ satisfies  $L_j^1('z,0)=\vert 'z\vert^2$. In the current coordinates, $\pa\Om$ is strictly pseudoconvex, and the complex tangent space to $\pa\Om$ at $p^j$ is given by $\{z_n=0\}$, thus the form $L_j^1('z,0)$ is strictly positive definite. Hence, there exists a unitary map $U_{p^j}:\mathbb C^{n-1}\rightarrow \mathbb C^{n-1}$ such that  $L_{j}^1(U_{p^j}('z),0)=\sum_{i=1}^n\la_i(p^j)|z_i|^2$, where $\la_i(p^j)>0,~i=1,\ldots,n-1$ are the eigenvalues of $L_j^1('z,0)$. Using the stretching map 
$R_{p^j}=\Big(z_1/\sqrt {\la_1(p^j)},\ldots,z_{n-1}/\sqrt{\la_{n-1}(p^j)}\Big)$, define the linear map \[A_{p^j}=R_{p^j}\circ U_{p^j},\] which satisfies 
$L_{p^j}^1(A_{p^j}('z),0)=|'z|^2$. Let $w=\phi_3^{p^j}(z)$ be defined as
\begin{equation}\label{phi_3-defn}
    \phi_3^{p^j}(z)=\Big(A_{p^j}('z), z_n\Big).
\end{equation}
Observe that 
\begin{equation}\label{conv-Aj}
    A_{p^j}\to\mathsf I_{n-1},
\end{equation}
as $j\to\infty$. Now, let $\mathsf Q_j=\psi_{p^j}'(\zeta^j)$ and as $j\to\infty$, note that
\begin{equation}\label{conv-Qj}
    \mathsf Q_j\to \mathsf I_{n},
\end{equation}
by (\ref{P-matrix}), (\ref{der-phi2-zetaj}) and (\ref{conv-Aj}).
Then, by relabeling $w$ as $z$, the local defining function $r\circ (\phi_1^{p^j})^{-1}\circ(\phi_2^{p^j})^{-1} (\phi_3^{p^j})^{-1}$ of $\Om_j^3=\psi_{p^j}(\Om)=\phi_3^{p^j}\circ\phi_2^{p^j}\circ\phi_1^{p^j}(\Om)$ near $0$ has the form as in (\ref{rj}).
Let $\Omega^0_j=\psi_{p^j}(\Omega^0),~ \Omega_j=\psi_{p^j}(\Omega)$, and $\eta_j = \de_{\Om_j^0}( q^j)=\de_{\Om_j}( q^j)=\de_{\Om_j'}( q^j)$. Then $ q^j = ('0, -\eta_j) $ where $\eta_j=\delta_j|\nabla_{\overline z}r(p^j)|,$ and hence by (\ref{normal-Re-z_n}),
\begin{equation}\label{eta-j/de-j}
    \frac{\eta_j}{\delta_j}=|\nabla_{\overline z}r(p^j)|\to |\nabla_{\overline z}r(0)|=1.
\end{equation}

Next, define the dilation maps $ T_{p^j} : \mathbb{C}^n \rightarrow \mathbb{C}^n $, by 
\begin{equation}\label{T_j-defn}
     T_{p^j}( 'z, z_n) = \left(   \frac{'z}{\sqrt{\eta_j}}, \frac{z_n}{{\eta}_j } \right).
\end{equation}
Note that 
\begin{equation}\label{det-T-matrix}
    \det\mathsf  T_{p^j}=\eta_j^{\frac{-(n+1)}{2}}.
\end{equation}
Let $\widetilde \Omega_j'=T_{p^j}(\Omega_j')$, $\widetilde\Om_j^0=T_{p^j}(\Omega^0)$, and $\widetilde{\Omega}_j=T_{p^j}(\Omega_j)$. We call $S_j=T_{p^j}\circ \psi_{p^j}$, the scaling maps. It follows that
\[S_j(\zeta^j)=T_{p^j}(q^j)=('0,-1)=b^*\in \widetilde \Omega_j'\subset\widetilde\Om_j^0\subset\widetilde \Omega_j,\]
and the scaled domains $ \widetilde \Om_j'= S_j(U\cap\Om)$ converge in the local Hausdorff sense to the unbounded realisation of the unit ball $\mf B^n$, namely to
\[
 \Om_{\infty}= \{ z \in \mathbb{C}^n : 2 \Re z_n + |'z|^2 < 0 \}.
\]
Recall that the Cayley transform
\begin{equation}\label{defn-H}
 H: ('z,z_n) \mapsto \left( \frac{{\sqrt{2}\;'z}}{1-z_n}, \frac{1+ z_n}{1 - z_n}\right)
\end{equation}
yields a biholomorphism from $ \Om _{\infty} $ onto $ \mathbb{B}^n$ with $H(b^*)=0$. Furthermore,
\begin{equation}\label{der-H}
H^{-1}=H,~ H'(b^*)=-\diag\{1/\sqrt 2,\ldots,1/\sqrt 2, 1/2\},~\det  H'(b^*)=(-1)^n2^{\frac{-(n+1)}{2}}.
\end{equation}
Define $\mathcal D_j'=H(\widetilde\Omega_j')$, $\mathcal D_j^0=H(\widetilde\Omega_j^0)$, and $\mathcal D_j=H(\widetilde \Omega_j)$. Then, we have
\begin{theorem}\label{ram-sgo}
    The sequence of Szeg\"o kernels $S_{\mathcal D_j}(z, w)$ converges to $S_{\mathbb B^n}(z, w)$ uniformly on compact subsets of $\mathbb B^n\times \mathbb B^n$, along with all the derivatives. 
\end{theorem}
\begin{theorem}\label{ram-sgo-loc}
    The sequence of Szeg\"o kernels $S_{\mathcal D_j^0}(z, w)$ converges to $S_{\mathbb B^n}(z, w)$ uniformly on compact subsets of $\mathbb B^n\times \mathbb B^n$, along with all the derivatives.
\end{theorem}

We require the following results to prove Theorem \ref{ram-sgo}.
\begin{lemma}\label{SK-inv}\cite[Theorem 1]{bl14}
    Let $F:\Omega_{1}\to\Omega_{2},~\Omega_{1},~\Omega_{2}\subset\mathbb{C}^{n}$ be a biholomorphic mapping.
Assume there exists a well-defined holomorphic branch of $\det\big(J_{\mathbb{C}}\big(F(z)\big)\big)^{\frac{n}{n+1}}$
on $\Omega_{1}.$ Then $SK_{\Omega}(z,w)=S_\Om^{n+1}(z,w)/K_\Om^n(z,w)$ is
invariant under biholomorphic mappings, i.e., \[SK_{\Omega_{1}}(z,w)=SK_{\Omega_{2}}\big(F(z),F(w)\big).\] 
Here, $K_\Om(z,w)$ denote the Bergman kernel on $\Om$.
\end{lemma}
\begin{lemma}\big(Bhatnagar \cite[Lemma 2.3]{b25}\big)\label{ram-berg}
    The sequence of Bergman kernels $K_{\mathcal D_j}(z, w)$ converges to $K_{\mathbb{B}^n}(z, w)$ uniformly on compact subsets of $\mathbb{B}^n\times \mathbb{B}^n$, along with all the derivatives. 
\end{lemma}

\begin{proof}[Proof of Theorem~\ref{ram-sgo}]
We first claim that $SK_{\mathcal{D}_j}$ is locally uniformly bounded on $\mathbb{B}^n$. To see this, let $V\Subset \mathbb B^n=H(\Om_\infty)$. For sufficiently large $j$, it follows that
  \[V\subset \mathcal D_j'\subset \mathcal D_j.\]
  Next, we observe that a well-defined holomorphic branch of 
  \[\det\big(J_{\mathbb{C}}\big(H\circ S_j(z)\big)\big)^{\frac{n}{n+1}} \]
  exits on $\Om$, as follows from (\ref{phi_1-defn}), (\ref{phi_2-defn}), (\ref{phi_3-defn}) and (\ref{defn-H}). Hence, by the invariance of $SK_\Om=S_\Om^{n+1}/K_\Om^n$ and \cite[Theorem 2]{bl14}, there exists a constant $M=M(\Om)>0$ such that for sufficiently large $j$,
\[SK_{\mathcal D_j}(z)=SK_\Omega\Big(S_j^{-1}\circ H(z)\Big)\leq M,\]
for all $z\in V$. 
  Consecutively, for $z, w\in V$, 
\begin{equation*}
    \left|S_{\mathcal D_j}(z,w)\right|\leq \sqrt{ S_{\mathcal D_j}(z)}\sqrt {S_{\mathcal D_j}(w)}\leq M\max_{z\in V}K_{V}(z)^n<\infty.
\end{equation*}
Therefore, by Montel's theorem, there exists a subsequence of $S_{\mathcal D_j}(z, w)$ that converges locally uniformly to a function, say $S_\infty(z, w)$ on $\mathbb{B}^n\times \mathbb B^n$. Moreover, by \cite[Lemma 2.4]{b25}, for a fixed $z_0\in\mathbb C^n$, $S_j^{-1}(z_0)$ converges to $0$, which yields as $j\to\infty$,
\begin{equation}\label{conv-SK_j}
    SK_{\mathcal D_j}(z)=SK_\Omega\Big(S_j^{-1}\circ H(z)\Big)\to\frac{(n-1)!}{c_n^{n+1}(n\pi)^n}=SK_{\mathbb B^n}(z),
\end{equation}
 for each $z\in\mathbb{B}^n$, using \cite[Theorem 2]{bl14}. Combining (\ref{conv-SK_j}) with Lemma \ref{ram-berg}, we get
\begin{multline*}
    S_\infty(z,z)=\lim_{j\to\infty}S_{\mathcal D_j}(z)=\lim_{j\to\infty}SK_{\mathcal D_j}(z)^{\frac{1}{n+1}}K_{\mathcal D_j}(z)^{\frac{n}{n+1}}\\
    = SK_{\mathbb B^n}(z)^{\frac{1}{n+1}}K_{\mathbb B^n}(z)^{\frac{n}{n+1}}=S_{\mathbb B^n}(z).
\end{multline*}
 Since the difference $S_{\infty}(z, \overline{w})-S_{\mathbb B^n}(z, \overline{w})$ is holomorphic in $ \mathbb B^n\times \mathbb B^n$ and vanishes along the diagonal, it follows that 
\[S_\infty(z, w)=S_{\mathbb{B}^n}(z, w).\]
Furthermore, the preceding arguments show that any convergent subsequence of $S_{\mathcal D_j}(z,w)$ has to converge to $S_{\mathbb B^n}(z,w)$ locally uniformly on $\mathbb B^n\times\mathbb B^n$, and therefore the entire sequence $S_{\mathcal D_j}(z,w)$ converges to $S_{\mathbb B^n}(z,w)$ locally uniformly on $\mathbb B^n\times\mathbb B^n$. Finally, the convergence of the derivatives follows from the harmonicity of $S_{\mathcal D_j}(z,w)$. This completes the proof.
\end{proof}
Then, we prove Theorem~\ref{ram-sgo-loc}.
\begin{proof}[Proof of Theorem~\ref{ram-sgo-loc}]
 For each $z\in\mathbb B^n$, it suffices to show that 
\[K_{\mathcal D_j^0}(z)\to K_{\mathbb B^n}(z),\]
as $j\to\infty$. Once this is established, the result follows by applying the same reasoning as in the proof of Theorem~\ref{ram-sgo}.

For each $0<s<1$, we have $\mathbb B^n(0,s)\Subset\mathbb B^n=H(\Om_\infty).$ Thus, for sufficiently large $j$, 
 \[\mathbb B^n(0,s)\subset \mathcal{D}_j'\subset \mathcal{D}_j^0.\] 
Then for $z, w\in \mathbb B^n(0,s/2)$,
 \[\vert K_{ \mathcal D_j^0}(z,w)\vert\leq \sqrt{K_{\mathbb B^n(0,s)}(z)}\sqrt{K_{\mathbb B^n(0,s)}(w)}\leq\max_{\mathbb B^n(0,s/2)} K_{\mathbb B^n(0,s)}(z).\]
 Therefore, $K_{ \mathcal D_j^0}(z,w)$ converges locally uniformly to a function, say $K_{\infty}(z, w)$ on $\mathbb{B}^n\times \mathbb{B}^n$. Now, for each $z\in\mathbb B^n$, we have
 \[1\leq\frac{K_{\Om^0}\Big(S_j^{-1}(H(z)\Big)}{K_{\Om}\Big(S_j^{-1}(H(z)\Big)}\leq \frac{K_{U\cap\Om}\Big(S_j^{-1}(H(z)\Big)}{K_{\Om}\Big(S_j^{-1}(H(z)\Big)},\]
 as $U\cap \Om\subset\widetilde\Om\subset\Om$. From \cite[lemma 2.4]{b25}, we obtain $S_j^{-1}(H(z))\to 0$ as $j\to\infty$, which implies
 \[\frac{K_{\mathcal{D}_j^0}(z)}{K_{\mathcal{D}_j}(z)}=\frac{K_{U\cap\Om}\Big(S_j^{-1}(H(z)\Big)}{K_{\Om}\Big(S_j^{-1}(H(z)\Big)}\to 1.\] 
Thus, by Lemma \ref{ram-berg},
 \[K_{\infty}(z,z)\equiv K_{\mathbb B^n}(z).\]
This establishes the sufficient condition, as required.
\end{proof}

\section{Boundary behaviour of invariants associated with the Fefferman--Szeg\"o metric}
Recall that $S_j=T_{p^j} \circ \psi_{p^j}$, $S_j( \Omega^0)=\widetilde \Omega_j^0$, $S_j(\Omega)=\widetilde \Omega_j$, $S_j(\zeta^j)=b^{*}=('0,-1)$, and 
\begin{equation}\label{der-S_j}
S_j'(\zeta^j)X=
\left(\frac{'(\mathsf{Q}_jX)}{\sqrt{\eta_j}}, \frac{(\mathsf{Q}_jX)_n}{\eta_j}\right),
\end{equation}
where $\mathsf Q_j=\psi_{p^j}'(\zeta^j)$, and $H:\Om_\infty\to\mathbb B^n$ is a biholomorphism with $H(b^*)=0$. In what follows, we will frequently use (\ref{conv-Qj}), (\ref{eta-j/de-j}), (\ref{det-T-matrix}), (\ref{der-H}), and (\ref{der-S_j}) without explicitly referring to them each time to prove Theorems \ref{bdy-n} and \ref{loc}. We also require the following lemmas.
\begin{lemma}\label{stability-sgo}
Let $D_j=\mathcal{D}_j\text{ or }\mathcal{D}_j^0$. Then, for $z \in \mathbb B^n$ and $X\in \mathbb{C}^n \setminus \{0\}$, the following limits hold as $j\to\infty$,
\begin{align*}
 g_{D_j}(z)\to g_{\mathbb B^n}(z), \quad \beta_{D_j}(z)\to \beta_{\mathbb B^n}(z),
\end{align*}
and also
\begin{align*}
& ds_{D_j}(z, X)\to ds_{\mathbb B^n}(z,X),  \quad  R_{D_j}(z, X)\to R_{\mathbb B^n}(z,X),\quad  \Ric_{D_j}(z, X)\to \Ric_{\mathbb B^n}(z,X).
\end{align*}
 Furthermore, the first and second convergences are uniform on compact subsets of \,$\mathbb B^n$ and the third, fourth and fifth convergences are uniform on compact subsets of\, $\mathbb B^n\times \mathbb C^n$.
\end{lemma}
\begin{proof}
  This is an immediate consequence of Theorems \ref{ram-sgo} and \ref{ram-sgo-loc}. 
\end{proof}

\begin{lemma}\label{comps-on-ball}
For the unit ball $\mathbb{B}^n \subset \mathbb{C}^n$, we have
\begin{itemize}
\item [(a)] $g_{\mathbb{B}^n}(z) =\dfrac{n^n}{(1-\vert z\vert^2)^{n+1}}$,
\item[(b)] $\beta_{\mathbb{B}^n}(z) =n^{n} \pi^{n+1}\left(\dfrac{c_n}{(n-1)!}\right)^{\frac{n+1}{n}}$,
\item[(c)] $ R_{\mathbb{B}^n}(z,X) =-\dfrac{2}{n},\quad X\in\mathbb{C}^n\setminus\{0\},$
\item[(d)] $\Ric_{\mathbb{B}^n}(z,X) =-\dfrac{1}{n},\quad X\in\mathbb{C}^n\setminus\{0\}.$
\end{itemize}
\end{lemma}
\begin{proof}
Recall from \cite{bl14} that
\[S_{\mf B^n}(z, w)=\frac{1}{c_n}\frac{(n-1)!}{\pi^n(1-z\cdot\ov w)^n},\]
where $z\cdot\ov w=\sum_{i=1}^nz_i\ov w_i$. This implies that
\[g_{\al\ov\be}=n\left(\frac{\delta_{\alpha\overline\beta}}{1-\vert z \vert^2}+\frac{\overline z_{\alpha}z_{\beta}}{(1-\vert z \vert^2)^2}\right) .\]
Thus,
\begin{equation*}
    g_{\mathbb{B}^n}(z)=\frac{n^n}{(1-\vert z \vert^2)^{n}}\Bigg(1+\frac{|z|^2}{(1-|z|^2)}\Bigg)=\frac{n^n}{(1-\vert z \vert^2)^{n+1}},
\end{equation*}
and hence
\[\beta_{ \mathbb{B}^n}(z)=\frac{g_{\mathbb{B}^n}(z)}{S_{\mathbb{B}^n}(z)^{\frac{n+1}{n}}}=\frac{\frac{n^n}{(1-\vert z \vert^2)^{n+1}}}{\left(\frac{1}{c_n}\frac{(n-1)!}{\pi^n(1-|z|^2)^n}\right)^{\frac{n+1}{n}}}=\left(\frac{c_n}{(n-1)!}\right)^{\frac{n+1}{n}}n^{n} \pi^{n+1}.\]
By performing the computations for the Fefferman--Szeg\"o metric, we have
\[R_{\mathbb{B}^n}(z, X)=-\frac{2}{n}\text{ and }\Ric_{\mathbb{B}^n}(z, X)=-1,\]
for any $X\in\mathbb{C}^n\setminus\{0\}$.
\end{proof}
We are now prepared to provide the proof of Theorem~\ref{bdy-n}.
\begin{proof}[Proof of Theorem~\ref{bdy-n}]
(a) By (\ref{tr-G}), we have
\[
g_{\Om}(p^j) = g_{\mathcal{D}_j}(0) \vert\det {H'}(b^*)\mathsf{T}_j\mathsf{Q}_j\vert^2 =\eta_j^{-(n+1)}g_{\mathcal D_j}(0)  \vert \det \mathsf{Q}_j\vert^{2}\big\vert \det H'(b^*)\big\vert^2.
\]
Thus,
\begin{multline*}
\de_j^{n+1} g_{\Om}(p^j) = \left(\frac{\de_j}{\eta_j}\right)^{n+1} g_{\mathcal D_j}(0)  \vert \det \mathsf{Q}_j\vert^{2}\big\vert \det H'(b^*)\big\vert^2\\
\to  g_{\mathbb{B}^n}(0)\big\vert \det H'(b^*)\big\vert^2= n^n\frac{1}{2^{n+1}} .
\end{multline*}

\medskip 

(b) $\be_{\Om}(p^j) = \be_{\mathcal{D}_j}(0) \to\be_{\mbb{B}^n}(0) = \left(\dfrac{c_n}{(n-1)!}\right)^{\frac{n+1}{n}}n^{n} \pi^{n+1}.$

\medskip 

(c) Using the invariance of Fefferman-Szeg\"o metric, we have
\begin{equation}\label{tau-v}
\tau_{\Om}(p^j,X)= \tau_{\mathcal{D}_j}\Big(0,H'(b^*)\mathsf{T}_j \mathsf{Q}_j X\Big)
= \tau_{\mathcal{D}_j}\left(0,H'(b^*)\left(\frac{'(\mathsf{Q}_jX)}{\sqrt{\eta_j}}, \frac{(\mathsf{Q}_jX)_n}{\eta_j}\right) \right).
\end{equation}
Therefore,
\begin{multline*}
\de_j\tau_{\Om}(p^j,X) =\frac{\de_j}{\eta_j} \tau_{\mathcal{D}_j}\Bigg(0,H'(b^*)\bigg(\sqrt{\eta_j}\,{'(\mathsf{Q}_jX)}, (\mathsf{Q}_jX)_n\bigg) \Bigg)  \to \tau_{\mbb{B}^n}\Big(0, H'(b^*) ('0,X_n) \Big)\\
= \tau_{\mbb{B}^n} \Big(0, ('0,-X_n/2)\Big)=\frac{1}{2}\sqrt{n}\big\vert X_N(0)\big\vert,
\end{multline*}
as $X_N(0)=({'}0,X_n)$ by \eqref{initial-df}.

\medskip 

(e) For simplicity, denote $X_H^j=X_H(p^j)$ and $X_H^0 \text{ as }X_H(p^0)=X_H(0)=({'X},0)$ by \eqref{initial-df}. We claim that $\big(\mathsf{Q}_j X_H^j\big)_n=0$. Indeed, recall that $\psi_{p^j}=\phi_3^{p^j}\circ\phi_2^{p^j}\circ\phi_1^{p^j},$
 which implies
\[
\mathsf{Q}_jX_H^j =\psi_{p^j}'(\z^j)= {R}_{p^j}\mathsf{U}_{p^j}\mathsf{P}_{p^j} X_H^j.
\]
From \eqref{phi1-n_z},
\begin{align*}
\mathsf{P}_{p^j} X_H^j & = \mathsf{P}_{p^j}\left(X-\left\langle X, \frac{\nabla_{\ov z} r(p^j)}{\vert \nabla_{\ov z} r(p^j)\vert}\right\rangle \frac{\nabla_{\ov z} r(p^j)}{\vert \nabla_{\ov z} r(p^j)\vert}\right)\\
&=\mathsf{P}_{p^j}X-\big\langle X, \nabla_{\ov z} r(p^j)\big\rangle \frac{1}{\vert \nabla_{\ov z} r(p^j)\vert^2}({'}0, \vert \nabla_{\ov z} r(p^j)\vert^2)\\
&=\mathsf{P}_{p^j}X-({'}0,\big\langle X, \nabla_{\ov z} r(p^j)\big\rangle)\\
& = \Big({'}\mathsf{P}_{p^j}X-{'}0, \big(\mathsf{P}_{p^j}X\big)_n- \big\langle X, \nabla_{\ov z} r(p^j)\big\rangle\Big)\\
&=({'\mathsf{P}_{p^j}X}, 0)\quad\text{ (by (\ref{P-matrix}))}.
\end{align*}
Since, $R_{p^j}$ and $\mathsf{U}_{p^j}$ 
  leave the $z_n$-coordinate unchanged, it follows that $\mathsf{Q}_jX_H^j=({'Y^j},0)$ for some ${'Y^j} \in \mbb{C}^{n-1}$, establishing the claim. From \eqref{tau-v}, we have
\[
\tau_{\Om}\big(p^j,X_H^j\big) =\tau_{\mathcal D_j}\left(0,H'(b^*)\left(\frac{'(\mathsf{Q}_jX_H^j)}{\sqrt{\eta_j}}, 0\right) \right).
\]
Thus,
\begin{align*}
\sqrt{\de_j}\tau_{\Om}\big(p^j,X_H^j\big)
&=\sqrt{\frac{\de_j}{\eta_j}} \tau_{\mathcal D_j}\Big(0,H'(b^*)\big({'(\mathsf{Q}_jX_H^j)},0\big) \Big) \to\\
&= \tau_{\mbb{B}^n}\big(0, H'(b^*)({'}X,0)\big) =  \tau_{\mbb{B}^n}\big(0, ({-'}X/\sqrt{2},0)\big)\\
& =\frac{1}{\sqrt{2}}\sqrt{n} \big\vert {'}X\big\vert\\
& =\sqrt{\frac{1}{2}n\mathcal{L}_\rho(0, X_H^0)},
\end{align*}
where the last equality follows from $\vert {'}X\vert^2 = \vert {'}X_H^0\vert^2=\mathcal{L}_\rho(0, X_H^0)$ by \eqref{initial-df}.

\medskip

(f) Let $\hat{X}= \lim_{j \to \infty} (H'(b^*)\mathsf{T}_j\mathsf{Q}_jX)/\vert H'(b^*)\mathsf{T}_j\mathsf{Q}_jX\vert$. Then
\begin{multline*}
R_{\Om}(p^j,X) =R_{\mathcal{D}_j}(b^*, H'(b^*)\mathsf{T}_j\mathsf{Q}_jX) = R_{\mathcal{D}_j}\left(b^*, \frac{H'(b^*)\mathsf{T}_j\mathsf{Q}_jX}{\vert H'(b^*)\mathsf{T}_j\mathsf{Q}_jX\vert}\right) \\ \to R_{\mf B^n}\big(b^*, \hat{X}\big)=-\frac{2}{n}.
\end{multline*}

\medskip 

(g) Proceeding as in (f), we obtain
\[
\Ric_{\Om}(p^j,X) \to \Ric_{\mf{B}^n}\big(0, H'(b^*)\hat{X}\big) = -1.
\]
Next, we prove Theorem \ref{loc}.
\begin{proof}[Proof of Theorem~\ref{loc}]
   The proof follows directly from Lemma \ref{stability-sgo} and the invariance of the Fefferman--Szeg\"o metric.
\end{proof}
\begin{remark}
    It would be interesting to prove Theorem \ref{ram-sgo} and, subsequently, Theorem \ref{ram-sgo-loc} without using the Bergman kernel.
\end{remark}
%%%%%%%%%%
\end{proof}

\end{document}